\documentclass[15pt]{article}
\usepackage[utf8]{inputenc}
\usepackage[american]{babel}
\usepackage{amsfonts, amsmath, amssymb}
\usepackage{mathtools}
\usepackage{fancyhdr}
\usepackage{amsthm}
\usepackage{graphicx}

\usepackage[
backend=biber,
style=alphabetic,
sorting=ynt
]{biblatex}

\newtheorem{thm}{Theorem}[section]
\newtheorem{conjecture}{Conjecture}[section]

\newtheorem{fact}{Fact}[section]

\newtheorem{manualtheoreminner}{Theorem}
\newenvironment{manualtheorem}[1]{%
  \renewcommand\themanualtheoreminner{#1}%
  \manualtheoreminner
}{\endmanualtheoreminner}

\theoremstyle{definition}
\newtheorem{definition}{Definition}[section]

\title{Bounds on Irrationality Measures and the Flint-Hills Series}
\author{Alex Meiburg\footnote{University of California Santa Barbara. Email: ameiburg@ucsb.edu}}

\begin{document}

\maketitle

\begin{abstract} 
It is unknown whether the Flint-Hills series $\sum_{n=1}^\infty \frac{1}{n^3\sin^2(n)}$ converges. Alekseyev (2011) connected this question to the irrationality measure of $\pi$, that $\mu(\pi) > \frac{5}{2}$ would imply divergence of the Flint-Hills series. In this paper we established a near-complete converse, that $\mu(\pi) < \frac{5}{2}$ would imply convergence. The associated results on the density of close rational approximations may be of independent interest. The remaining edge case of $\mu(\pi) = \frac{5}{2}$ is briefly addressed, with evidence that it would be hard to resolve.
\end{abstract}

\section{Introduction}
The Flint-Hills series is the series
$$\sum_{n=1}^\infty \frac{1}{n^3\sin^2(n)}$$
and the question of its convergence remains open (Pickover~\cite{Pickover2007}). Alekseyev~\cite{Alekseyev} showed that convergence would imply that $\mu(\pi) \le \frac{5}{2}$, where $\mu(\pi)$ is the irrationality measure of $\pi$. His argument is easily summarized: if $\mu(\pi)$ is too large, then multiples of $\pi$ are well approximated by integers, $\sin(n)$ becomes exceedingly small, and the individual terms of the sequence grow without bound. The converse is not immediate however, as a small $\mu(\pi)$ leads to a sequence whose terms converge to zero but summed series may still grow without bound.
\par Our main result is a near-complete converse, that if $\mu(\pi) < \frac{5}{2}$, then the Flint-Hills series converges. This requires inspecting not only the size of the terms, but how {\em frequently} good approximations of $\pi$ can occur. The associated results concerning the density of good approximations are elementary to derive and likely known, but the author could not find reference to them in prior literature. They are similar to the bounds in Korhonen~\cite{Korhonen}, where it is shown that the denominators of quadratically good convergents must grow exponentially. Our results refine this in terms of different exponents on the denominator and the irrationality measure. This allows us to establish that the close approximations to $\pi$ are sparse enough that the series must converge.
\par It is not a complete converse, as the edge case $\mu(\pi) = \frac{5}{2}$ is still consistent with the Flint-Hills series converging or not. We provide evidence that this case cannot be resolved through analysis of the irrationality measure alone: we believe there are two numbers $\alpha_1$, $\alpha_2$ with both $\mu(\alpha_i) = \frac{5}{2}$ such that
$$\sum_{n=1}^\infty \frac{1}{n^3\sin^2(n \alpha_i / \pi)}$$
converges for $i=1$ and diverges for $i=2$. In other words, if we replace the period $\pi$ with a different irrational number, that the series can converge or diverge, and the irrationality measure of the period is not enough information to determine which. We prove in Theorem \ref{thm:canDiv} that at least $\alpha_2$ exists. But, in the likely event that $\mu(\pi)$ is eventually bounded below $\frac{5}{2}$, this edge case would become irrelevant, and together with our result would imply that the Flint-Hill series converges.

\section{Techniques}
We begin by reviewing the definition of irrationality measure:
\begin{definition}[Irrationality measure]
Given $\alpha\in\mathbb{R}$, the irrationality measure $\mu(\alpha)$ is
$$\mu(\alpha) = \inf_{\mu_0 \in \mathbb{R}^+} \Bigg[ 0 < \left| \alpha-\frac{p}{q}\right|<\frac{1}{q^{\mu_0}} \textrm{  has infinitely many solutions  } (p,q). \Bigg]$$
\end{definition}

Consequently, for any exponent $s > \mu(\alpha)$, there will be only finitely many solutions to $|\alpha - p/q| < 1/q^s$. To a rational expression $(p,q)$ approximating a real $\alpha$, we associate its {\em approximation exponent} $s$,
$$s = -\log_q(p/q - \alpha),$$
and to each integer $q$ the minimum approximation exponent across any $(p,q)$.
\par For a small $\epsilon > 0$, there will be infinitely many approximations with $s > \mu(\alpha) - \epsilon$, but we expect these to occur rarely. We call these approximations {\em good}:

\begin{definition}[$\epsilon$-good approximation]
Given $p,q\in\mathbb{Z}$ and $\alpha\in\mathbb{R}$, we say the approximation $\frac{p}{q}$ is an $\epsilon$-good approximation of $\alpha$ if
$$\left | \alpha-\frac{p}{q}\right|<\frac{1}{q^{\mu(\alpha)-\epsilon}}.$$
\end{definition}

Sometimes we will simply refer to $q$ as an approximation, in which case we are implicitly taking $p$ to be the integer minimizing $|\alpha - p/q|$. This leads to the related notion of the {\em approximation exponent} of each integer $q$:

\begin{definition}[Approximation exponent]
The approximation exponent of a denominator $q$ relative to an irrational $\alpha$ is
\begin{equation}
    r(q) = -\log_q\left( \min_p \left|\frac{1}{\alpha}-\frac{p}{q}\right| \right) = 1-\log_q\left( \frac{1}{\alpha} \min_p \left|q-\alpha p\right| \right)
\end{equation}
\end{definition}

Our first goal will be to show that good approximations do not cluster close together, in a particular sense:

\begin{definition}[$\delta$-close approximations]
Given $q_1,q_2\in\mathbb{Z}$, $q_1<q_2$, we say that $(q_1, q_2)$ are $\delta$-close if
$$q_2-q_1<q_1^\delta.$$
\end{definition}

The first result is a bound on the density of $\epsilon$-approximations, that is, approximations that come close to the irrationality measure $\mu(\alpha)$. They cannot occur with high density, or conversely, the sequence of $q$ with exponent at least $\mu(\alpha)-\epsilon$ must grow quickly.
\begin{thm}[Good approximations are rare]\label{thm:a}
Suppose that positive reals $\alpha$, $\epsilon_1$, and $\epsilon_2$ satisfy
$$0 < \epsilon_2 < 1, \quad \mu(\alpha) > 1 + \frac{\epsilon_1}{1-\epsilon_2}.$$
Let $Q=\left\{q\in\mathbb{Z}^+|\exists p\in\mathbb{Z}^+: 0<\left | \alpha-\frac{p}{q}\right |<\frac{1}{q^{\mu(\alpha)-\epsilon_1}}\right\}$ be the sequence of $\epsilon_1$-good approximations to $\alpha$. Then $Q_n$, the $n$th element of $Q$, grows as $\Omega\left(n^{\frac{1}{1-\epsilon_2}}\right)$ with a constant that depends only on $\alpha$ and $\epsilon_1$.
\end{thm}

This bound will allow us to control the contribution of these terms in the Flint-Hill series. There are only a few essential properties of the $|\sin(x)|$ function that we need. We can discuss a generalized Flint-Hill series, built from a function $\mathcal{P}:\mathbb{R}^+\to\mathbb{R}^+$ with the properties (given an irrational period $\alpha$, and bounds $B_1$ and $B_2$):
\begin{enumerate}
  \item $\mathcal{P}(x)=\mathcal{P}(x+n\alpha)$, for all $x\in\mathbb{R}^+, n\in\mathbb{N}$
  \item If $|x|\le \alpha/2$, then $B_1x \le \mathcal{P}(x) \le B_2x$.
\end{enumerate}
We call a function with these two properties ``sine-like``. Then given exponents $u$ and $v$, the sequence is given by $\mathcal{A}_{u,v}(n) = n^{-u}\mathcal{P}(n)^{-v}$, and its partial sums the series $\mathcal{S}_{u,v}(n) = \sum^{n}_{i=1}\mathcal{A}_{u,v}(i)$. The Flint-Hills series fits this form with $\alpha=\pi$, $B_1=\frac{1}{2}$, $B_2 = 1$, $\mathcal{P}=|\sin|$, $u=3$, and $v=2$. \cite{Alekseyev} gave necessary conditions for the convergence of $\mathcal{S}$:

\begin{thm}[\cite{Alekseyev}1, Corollary 3]
For $u,v>0$, if $A_{u,v}$ converges, then $\mu(\alpha) \le 1 + u/v$. If $\mathcal{A}_{u,v}$ diverges, then $\mu(\alpha) \ge 1 + u/v$.
\end{thm}
\begin{thm}[\cite{Alekseyev}, Corollary 4]
For $u,v>0$, if $\mathcal{S}_{u,v}$ converges, then $\mu(\alpha) \le 1 + u/v$.
\end{thm}
as for $S$ to converge, $A$ must converge. The only sufficient condition given is the comparatively weak
\begin{thm}[\cite{Alekseyev}, Theorem 5]
For $u,v>0$, if $\mu(\alpha) < 1 + (u-1)/v$, then $\mathcal{S}_{u,v}$ converges.
\end{thm}
which in the case of the Flint-Hill series (where $u=3$, $v=2$), gives an (unfulfilled) sufficient condition that $\mu(\pi) < 2$. Using Theorem \ref{thm:a}, we will perform a more careful analysis and give a tight sufficient condition.

\begin{thm}[Main result]\label{thm:main}
For any sine-like function $\mathcal{P}$ with irrational period $\alpha$, constant $v \ge 1$, and $\mu(\alpha) < 1 + \frac{u}{v}$, the series $\mathcal{S}_{u,v}(n)$ converges.
\end{thm}

In order to show that the series $\sum^{\infty}_{q=1}\mathcal{A}_{u,v}(q)$ converges, we will partition the integers $q$ into different sets based on how accurately they can approximate $1/\alpha$. Before proving the general case, we will demonstrate it with a partitioning of the series into three sets, based on their approximations exponents. This will already give us a nontrivial sufficient condition on the convergence, but not the best condition possible. For the full proof in the next section, we will generalize the technique to finer partitions, and prove the full theorem.

\begin{fact}[Weaker main result]\label{thm:weak}
For any sine-like function $\mathcal{P}$ with irrational period $\alpha$, constant $v \ge 1$, and
\begin{equation}
    \mu(\alpha) < \frac{\sqrt{(u+3)(u-1)} + u-1}{2v} + 1.
\end{equation}
the series $\mathcal{S}_{u,v}(n)$ converges.
\end{fact}
\begin{proof}
Each term $\mathcal{A}_{u,v}(n)=\frac{1}{n^{u}\mathcal{P}(n)^{v}}$ in the series can be bounded based on its approximation exponent:
\begin{align}
\mathcal{P}(n)&=\mathcal{P}(n-\alpha m)\\
&\geq B_1 \min_m |n-\alpha m|\\
&= B_1 \alpha n^{1-r(n)}.
\end{align}
\begin{align}\label{eqn:ABound}
\mathcal{A}_{u,v}(n)&=\frac{1}{n^{u}\mathcal{P}(n)^{v}}\\
&\leq\frac{1}{(\alpha B_1)^v}\frac{1}{n^{u+v-vr(n)}}\\
&=\mathcal O\left(\frac{1}{n^{u+v-vr(n)}}\right)
\end{align}

with a constant that depends on $\mathcal{P}$ and $v$ but not $n$. Partition the integers into three sets based on their approximation exponents to $1/\alpha$, and two positive reals $x$ and $y$ that remain to be chosen:
$$S_1=\left\{q\in\mathbb{Z}^+\Big|r(q) \ge \mu(\alpha)+y\right\}$$
$$S_2=\left\{q\in\mathbb{Z}^+\Big|r(q)\in[\mu(\alpha)-x,\mu(\alpha)+y)\right\}$$
$$S_3=\left\{q\in\mathbb{Z}^+\Big|r(q)<\mu(\alpha)-x\right\}.$$
keeping in mind that $\mu(\alpha) = \mu(1/\alpha)$. The sum $\mathcal{S}$ is decomposed by the partition,
\begin{align}
\mathcal{S}_{u,v}&=\sum^{\infty}_{i=1}\mathcal{A}_{u,v}(i)\\
&=\sum_{q\in S_1}\mathcal{A}_{u,v}(q)+\sum_{q\in S_2}\mathcal{A}_{u,v}(q)+\sum_{q\in S_3}\mathcal{A}_{u,v}(q)\label{eqn:part3}.
\end{align}
As $y>0$, the set $S_1$ is finite, and the first term converges. As for the second term, each index $q$ obeys $r(q) < \mu(\alpha) + y$
so that
$$\sum_{q\in S_2}\mathcal{A}_{u,v}(q) \leq \frac{1}{(\alpha B_1)^v}\sum_{q\in S_2} \frac{1}{q^{u+v-vr(n)}} \leq \frac{1}{(\alpha B_1)^v}\sum_{q\in S_2} q^{v(\mu(\alpha)+y)-u-v}$$

Now we can apply Theorem \ref{thm:a} to $1/\alpha$, $\epsilon_1 = x$, and with and some (as-of-yet undetermined) value of $\epsilon_2 < 1 - \frac{x}{\mu(\alpha) -  1}$. Then the sequence $q \in S_2$ grows at least as quickly as $\Omega(n^{1/(1-\epsilon_2)})$, and the second sum is bounded by

\begin{align*}
\sum_{j\in S_2}\mathcal{A}_{u,v}(j) &\leq \frac{1}{(\alpha B_1)^v}\sum_{q\in S_2} q^{v(\mu(\alpha)+y)-u-v}\\
&= \frac{1}{(\alpha B_1)^v}\sum_{n=1}^{\infty}\Omega\left(n^\frac{1}{1-\epsilon_2}\right)^{v(\mu(\alpha)+y)-u-v}\\
&\le C\sum_{n=1}^{\infty}n^{\frac{v(\mu(\alpha)+y-1)-u}{1-\epsilon_2}}
\end{align*}
which converges provided that $\frac{v(\mu(\alpha)+y-1)-u}{1-\epsilon_2} < -1$, or that $\epsilon_2 > 1-u-v(1-\mu(\alpha)-y)$. Thus we want to choose $\epsilon_2$ such that
$$1 - \frac{x}{\mu(\alpha) -  1} > \epsilon_2 > 1-u-v(1-\mu(\alpha)-y)$$
which is possible exactly when
$$1 - \frac{x}{\mu(\alpha) -  1} > 1-u-v(1-\mu(\alpha)-y)$$
\begin{equation}\label{eqn:e3cond}
\iff x < (u+v(1-\mu(\alpha)-y))(\mu(\alpha)-1).\end{equation}
Under this conditions, we can guarantee the second term in Eqn \ref{eqn:part3} will converge.

This leaves the third term of Eqn \ref{eqn:part3}. We again apply the bound on $\mathcal{A}$ from Eqn \ref{eqn:ABound}, but this time with $r(q) < \mu(\alpha)-x$:
\begin{align*}
\sum_{q\in S_3}\mathcal{A}_{u,v}(q) &=\sum_{q\in S_3} \frac{C}{q^{u+v-vr(q)}}\\
&\le \sum_{q\in S_3}\frac{C}{q^{u+v-v(\mu(\alpha)-x)}}\\
&\le \sum_{q\in \mathbb{N}} \frac{1}{q^{u+v-v(\mu(\alpha)-x)}}.\\
\end{align*}
This sum converges when $u+v-v(\mu(\alpha)-x) > 1$, or when $x > \mu(\alpha) + \frac{1-u}{v} - 1$. This means the series $\mathcal{S}$ must converge whenever we have a simultaneous solution to the four inequalities,
\begin{eqnarray*}
x > 0, \quad y>0\\
x < (u+v(1-\mu(\alpha)-y))(\mu(\alpha)-1)\\
x > \mu(\alpha) + \frac{1-u}{v} - 1
\end{eqnarray*}
Simple algebraic manipulation shows this has a solution exactly when
\begin{equation}
    2 \le \mu(\alpha) < \frac{\sqrt{(u+3)(u-1)} + u-1}{2v} + 1.
\end{equation}
\end{proof}

In the Flint-Hills case where $u=3$, $v=2$, this means the sum converges whenever $\mu(\pi) < \frac{3+\sqrt{3}}{2} \approxeq 2.366$. This result required partitioning the integers into a set $S_1$ (finitely many $q$'s that are ``too good" of an approximation, with $r(q)$ bounded away $\mu(\alpha)$ from above), a set $S_3$ ($q$'s that are bad approximations, $r(q)$ at least $\frac{1-u-v}{v}$ less than $\mu(\alpha)$, and can't diverge no matter how dense they are), and an intermediate set $S_2$ that produce large terms in $\mathcal{A}$ but infrequently. We would like to tighten this to match Alekseyev's bound of $\mu(\alpha) = 1+u/v$, and this requires refining the partition $S_2$, so that terms with better approximations are appropriately less frequent. In the next section, we prove Theorem~\ref{thm:a}, and apply the above approach to finer partitions.

\section{Proofs}

We now prove that good approximations are not frequently close.

\begin{manualtheorem}{\ref{thm:a}}[Good approximations are rare]
Suppose that positive reals $\alpha$, $\epsilon_1$, and $\epsilon_2$ satisfy
$$0 < \epsilon_2 < 1, \quad \mu(\alpha) > 1 + \frac{\epsilon_1}{1-\epsilon_2}.$$
Let $Q=\left\{q\in\mathbb{Z}^+|\exists p\in\mathbb{Z}^+: 0<\left | \alpha-\frac{p}{q}\right |<\frac{1}{q^{\mu(\alpha)-\epsilon_1}}\right\}$ be the sequence of $\epsilon_1$-good approximations to $\alpha$. Then $Q_n$, the $n$th element of $Q$, grows as $\Omega\left(n^{\frac{1}{1-\epsilon_2}}\right)$ with a constant that depends only on $\alpha$ and $\epsilon_1$.
\end{manualtheorem}
\begin{proof}
Take two $\epsilon_1$-good approximations to $\alpha$, $(p_1,q_1)$ and $(p_2,q_2)$, where $q_1$ and $q_2$ are $\epsilon_2$-close. Let $A,B$ be the respective errors of the approximations:
$$A=\frac{p_1}{q_1}-\alpha$$
$$B=\frac{p_2}{q_2}-\alpha$$
Then we can rearrange:
\begin{align*}
q_2B-q_1A&=(p_2-\alpha q_2)-(p_1-\alpha q_1)\\
&=(p_2-p_1)-(q_2-q_1)\alpha
\end{align*}
$$\implies \frac{p_2-p_1}{q_2-q_1}-\alpha=\frac{q_2B-q_1A}{q_2-q_1}.$$
Since $A$ and $B$ are both small, this suggests that $\frac{p_2-p_1}{q_2-q_1}$ could be a good approximation to $\alpha$ as well. Both of the initial approximations to $\alpha$ are $\epsilon_1$-good, so
$$\left|q_1A\right|=q_1|A|<\frac{q_1}{q_1^{\mu(\alpha)-\epsilon_1}}=\frac{1}{q_1^{\mu(\alpha)-\epsilon_1-1}},$$
$$\left|q_2B\right|=q_2|B|<\frac{q_2}{q_2^{\mu(\alpha)-\epsilon_1}}<\frac{1}{q_1^{\mu(\alpha)-\epsilon_1-1}}.$$
Then the error $E$ of the new approximation is bounded,
\begin{align*}
E=\frac{\left|q_2B-q_1A\right|}{q_2-q_1}&\leq\frac{2q_1^{-\mu(\alpha)+\epsilon_1+1}}{q_2-q_1},
\end{align*}
so that it is an approximation with exponent at least as high as $-\log_{(q_2-q_1)}\left(\frac{2q_1^{-\mu(\alpha)+\epsilon_1+1}}{q_2-q_1}\right)$. This is then bounded below by
\begin{align*}
-\log_{(q_2-q_1)}\left(\frac{2q_1^{-\mu(\alpha)+\epsilon_1+1}}{q_2-q_1}\right)
&=1-\frac{(-\mu(\alpha)+\epsilon_1+1)\log\left(q_1\right)+\log\left(2\right)}{\log(q_2-q_1)}\\
&>1+\frac{(\mu(\alpha)-\epsilon_1-1)\log\left(q_1\right)-\log\left(2\right)}{\log(q_1^{\epsilon_2})}\\
&=1+\frac{\mu(\alpha)-1-\epsilon_1}{\epsilon_2}-\frac{1}{\epsilon_2 \log_2(q_1)}\\
\end{align*}

This exponent will exceed $\mu(\alpha)$ when:
\begin{align*}
\mu(\alpha)&<1+\frac{\mu(\alpha)-1-\epsilon_1}{\epsilon_2}-\frac{1}{\epsilon_2 \log_2(q_1)}\\
\iff\mu(\alpha)\left(1-\frac{1}{\epsilon_2}\right)&<1-\frac{1+\epsilon_1}{\epsilon_2}-\frac{1}{\epsilon_2 \log_2(q_1)}\\
\iff\mu(\alpha)&>\left(\frac{-\epsilon_2}{1-\epsilon_2}\right)\left(1-\frac{1+\epsilon_1}{\epsilon_2}-\frac{1}{\epsilon_2 \log_2(q_1)}\right)\\
 &=1+\frac{\epsilon_1}{1-\epsilon_2}+\frac{1}{(1-\epsilon_2) \log_2(q_1)}
\end{align*}

This will hold for sufficiently large $q_1$, given that $\epsilon_2 < 1$ (so that third term becomes arbitrarily small), and given that
\begin{equation}\label{eq:epsCond}
    \mu(\alpha) > 1 + \frac{\epsilon_1}{1-\epsilon_2}.
\end{equation}
These two conditions are the hypotheses of the theorem. There is some finite $q_0$ such that
$$\mu(\alpha) > 1+\frac{\epsilon_1}{1-\epsilon_2}+\frac{1}{(1-\epsilon_2) \log_2(q_0)}.$$
Denote the slack in this inequality by
$$L = \mu(\alpha) - \left(1+\frac{\epsilon_1}{1-\epsilon_2}+\frac{1}{(1-\epsilon_2) \log_2(q_0)}\right) > 0.$$
So, let's suppose we are given a pair of $\epsilon_1$-good, $\epsilon_2$-close approximations $(p_1, q_1)$ and $(p_2,q_2)$ with $q_2 > q_1 > q_0$. From this, we can construct a new approximation $\frac{p_2 - p_1}{q_1 - q_1}$ with an approximation exponent of at least
$$1+\frac{\mu(\alpha)-1-\epsilon_1}{\epsilon_2}-\frac{1}{\epsilon_2 \log_2(q)} > 1+\frac{\mu(\alpha)-1-\epsilon_1}{\epsilon_2}-\frac{1}{\epsilon_2 \log_2(q_0)} $$
$$= \mu(\alpha) + \frac{L(1-\epsilon_2)}{\epsilon_2} > \mu(\alpha).$$
Since this is a value strictly above $\mu(\alpha)$, there can only be finitely many distinct approximations of the form $(p_2-p_1, q_2-q_1)$ satisfying our assumptions of goodness and closeness on $(p_1, q_1)$ and $(p_2,q_2)$. Significantly, among the approximations $\frac{p_2-p_1}{q_2-q_1}$, there is some largest denominator $Q_{\max}$. 

However, this doesn't immediately imply that there are only finitely many good close pairs of approximations $(p_1,q_1),(p_2,q_2)$, as it could be that they produce the same pair $(p_2-p_1, q_2-q_1)$ infinitely many times. That is, there can still be infinitely many pairs of $\epsilon_1$-good approximations of $\alpha$ with $\epsilon_2$-close denominators, if the same gap $q_2-q_1$ approximation occurs infinitely times.

We finally turn to $\mathcal{Q}$, the set of denominators to $\epsilon_1$-good approximations:
$$\mathcal{Q}=\left\{q\in\mathbb{Z}^{+}|\exists p\in\mathbb{Z}^+, 0<\left|\alpha-\frac{p}{q}\right|<\frac{1}{q^{\mu(\alpha)-\epsilon_1}}\right\}$$
Any two numbers in $\mathcal{Q}$ are either $\epsilon_2$-close or not. If they are, then we must have $q_2-q_1 \le Q_{\max}$. Otherwise, they are not close, and they are at least $q_1^{\epsilon_2}$ apart. This means that after any given $q_1$, the intersection $\mathcal{Q} \cap (q_1, q_1+q_1^{\epsilon_2}]$ has at most $Q_{\max}$ many elements. This is an upper bound on the density of the denominators of $\epsilon_1$-good approximations, and conversely a lower bound on their growth rate. Each interval $(q_1,q_1+q_1^{\epsilon_2}]$ contains only finitely many elements, meaning that the $n$th denominator grows as $\Omega\left(n^{\frac{1}{1-{\epsilon_2}}}\right)$.
\end{proof}

We can now give the proof of our main result.

\begin{manualtheorem}{\ref{thm:main}}[Main Result]
For any sine-like function $\mathcal{P}$ with irrational period $\alpha$, constant $v \ge 1$, and $\mu(\alpha) < 1 + \frac{u}{v}$, the series $\mathcal{S}_{u,v}(n)$ converges.
\end{manualtheorem}
\begin{proof}
 As above, we choose $x$ and $y$ as to make the sums $S_1$ and $S_3$  finite, that is, $y>0$ and $x>\mu(\alpha) + \frac{1-u}{v} - 1$. Then, we will partition the set $S_2$ into $k$ smaller sets:
\begin{align*}
T_1=\left\{q\in\mathbb{Z}^+\Big|r_\alpha(q)\in[\mu(\alpha)-x,a_1)\right\}\\
T_2=\left\{q\in\mathbb{Z}^+\Big|r_\alpha(q)\in[a_1,a_2)\right\}\\
T_3=\left\{q\in\mathbb{Z}^+\Big|r_\alpha(q)\in[a_2,a_3)\right\}\\
\dots \\
T_k=\left\{q\in\mathbb{Z}^+\Big|r_\alpha(q)\in[a_{k-1},\mu(\alpha)+y)\right\}
\end{align*}
where $k$ and $a_i$ are values to be determined. To make the description more uniform, we will denote $a_0 = \mu(\alpha)-x$ and $a_k = \mu(\alpha)+y$. Apply Theorem \ref{thm:a} to $T_i$ with $\epsilon_1 = \mu(\alpha) - a_{i-1}$ and some $\epsilon_2 = b_i > 0$. Then the set $T_i$ grows as $\Omega(n^{1/(1-b_i)})$ for any
$$b_i < 1 - \frac{\mu(\alpha)-a_{i-1}}{\mu(\alpha)-1},$$
and the sum will converge as long as the exponent in the sum
$$\left(u+v-v a_i\right)\left(\frac{1}{1-b_i}\right) > 1$$
This has a solution in $b_i$ exactly when
\begin{eqnarray}
    u+v-v a_i > \frac{\mu(\alpha)-a_{i-1}}{\mu(\alpha)-1}\\
    \iff a_i < 1 + \frac{u-1}{v} + \frac{a_{i-1}-1}{v(\mu(\alpha)-1)}.
\end{eqnarray}
This limits the size of the interval, $a_i - a_{i-1}$:
\begin{eqnarray}
       & a_i - a_{i-1}\\
     < & (1 + \frac{u-1}{v} + \frac{a_{i-1}-1}{v(\mu(\alpha)-1)}) - a_{i-1}\\
     = & \frac{a_{i-1}(1+v-v\mu(\alpha)) + (\mu(\alpha)-1)(u+v-1) -1}{v(\mu(\alpha)-1)} \label{eqn:bb}\\
     = & f(a_{i-1})
\end{eqnarray}
There will be a sequence of $a_i$ that produce finite-sum partitions, as long as the $f$ in Eqn \ref{eqn:bb} is bounded away from zero on the interval $[a_0,a_k]$. $f$ is linear in its argument, with derivative
$$f' = 1 - \frac{1}{v(\mu(\alpha)-1)}$$
Since $\mu(\alpha) \ge 2$, as long as $v\ge 1$, we know that $f'$ is nonpositive, thus its minimum value is attained at the right end of the interval, $a_k$, and $a_k$ can be chosen arbitrarily close to $\mu(\alpha)$. Evaluated at that point,
\begin{eqnarray}
    f(\mu(\alpha)) &= \frac{\mu(\alpha)(1+v-v\mu(\alpha)) + (\mu(\alpha)-1)(u+v-1) -1}{v(\mu(\alpha)-1)}\\
    &= 1 + \frac{u}{v} - \mu(\alpha)\label{eqn:cc}
\end{eqnarray}
So, in summary: as long as $2 \le \mu(\alpha) < 1 + \frac{u}{v}$ and $v \ge 1$, the linear function $f$ is nonincreasing, and is bounded above zero at the right point of the interval $[\frac{1-u}{v} - 1, \mu(\alpha)]$. This means that we can pick values $a_0$ and $a_k$ sufficiently close to the ends of the interval so that they are also positive: $f(a_0) > C$, $f(a_k) > C$, for some positive $C$. Then we can partition that interval into finitely many intervals of width at most $C$. Each interval individually contributes terms to $\mathcal{S}_2$ which have a finite sum, so $\mathcal{S}_2$ is a finite sum. Thus the whole sum $\mathcal{S}_{u,v}$ converges.
\end{proof}

This result is an almost complete converse of Alekseyev's Theorem 4, as he has showed that $\mu(\pi) \le 1 + u/v$ is necessary, and we have showed that $\mu(\pi) < 1 + u/v$ is sufficient. In the event that $\mu(\pi)=1+u/v=5/2$, this will leave unanswered whether or not the Flint Hill series converges. We can show that our bound is in some sense the best possible: there are irrational numbers $\alpha$ with $\mu(\alpha) = 1+u/v$ such that the corresponding series can diverge.

\begin{thm}\label{thm:canDiv}
For any pair $u, v> 0$, there is an irrational $\alpha$ with $\mu(\alpha)=1+\frac{u}{v}$, such that for any sine-like function $\mathcal{P}$ with period $\alpha$, both $\mathcal{A}_{u,v}$ and $\mathcal{S}_{u,v}$ diverge.
\end{thm}
\begin{proof}
We construct $\alpha$ by giving the continued fraction expansion of its reciprocal $1/\alpha = [a_0; a_1, a_2...]$, which determines its irrationality measure (Sondow~\cite{Sondow04}):
$$\mu(\alpha)=\mu(1/\alpha)=1+\limsup_{n\to\infty}\frac{\ln(q_{n+1})}{\ln(q_n)}=2+\limsup_{n\to\infty}\frac{\ln(a_{n+1})}{\ln(q_n)}$$
where $p_n/q_n$ are the convergents of $1/\alpha$. Given a partial continued fraction expansion $[a_0 ; a_1 \dots a_k]$, we will repeatedly extend it such that $\limsup_n \mathcal{A}_{u,v}(n) = 1$ and thus the series diverges, while achieving the desired $\mu(\alpha)$.
\par If $a_{n+1}$ is an exceptionally large integer, then we can conclude that the convergent $p_n/q_n$ is an exceptionally good approximation, and leads to a large term $\mathcal{A}_{u,v}(n)$. The error is bounded by
$$\frac{1}{(a_{n+1}+2)q_n^2} < \left| \frac{1}{\alpha} - \frac{p_n}{q_n} \right| < \frac{1}{a_{n+1}q_n^2}$$
or, in terms of the approximation exponent,
$$\frac{1}{(a_{n+1}+2)q_n^2} < q^{-r(q)} < \frac{1}{a_{n+1}q_n^2}$$
so that
$$\mathcal{P}(n) = \mathcal{P}(n-\alpha m) \le B_2 \alpha n^{1-r(n)}$$
$$\mathcal{A}_{u,v}(q_n) \ge \frac{1}{(\alpha B_2)^v}\frac{1}{q_n^{u+v-vr(q_n)}} > \frac{1}{(\alpha B_2)^v}\frac{1}{q_n^{u+v}}\left(a_{n+1}q_n^2\right)^v$$
This implies that $\mathcal{A}_{u,v}(q_n) > 1$ whenever
$$a_{n+1} \ge \alpha B_2 q_n^{u/v-1}$$
Taking $a_{n+1}$ to be the ceiling of the value on the right,
$$\mu(\alpha)=2+\limsup_{n\to\infty}\frac{\ln(\alpha B_2)+\ln( q_n^{u/v-1})}{\ln(q_n)} = 2 + (u/v-1) = 1 + u/v$$
as desired. By continuing this infinitely (from some prefix fraction, say $[0;1]$) we can cause $\mathcal{S}$ to diverge, as it contains infinitely many terms at least 1.
\end{proof}
The author believes that convergence can also occur:
\begin{conjecture}
For any pair $u, v> 0$, with $1 + \frac{u}{v} > 2$, there is an irrational $\alpha$ with $\mu(\alpha)=1+\frac{u}{v}$, such that for any sine-like function $\mathcal{P}$ with period $\alpha$, the series $\mathcal{S}_{u,v}$ converges (and thus $\mathcal{A}_{u,v}$ as well).
\end{conjecture}
but proving this would likely involve repeating an analysis similar to the main Theorem \ref{thm:main}, but more carefully and constructively the whole way through. Theorem \ref{thm:canDiv} establishes, at least, that any proof that the original Flint-Hills series converges would need some property of $\pi$ besides its irrationality measure alone.

\section{Conclusion}
We established a near converse to the work of Alekseyev. It is widely suspected that $\mu(\pi)=2$, and progress such as that of Zeilberger and Zudilin~\cite{Zeilberger} has showed that $\mu(\pi) \le 7.104\dots$, with this bound steadily decreasing over time. If at some point it is established that $\mu(\pi) \le \frac{5}{2}$, then the convergence Flint-Hills series will be solved. We also established some elementary results on the density of good rational approximations to irrationals in terms of their irrational measures.

\section{Acknowledgements}
This paper would be incomplete without a sincere thank you to Adam Wang, who initially directed the author to this problem, and whose discussions inspired the author to work on this.

\printbibliography

\end{document}